\documentclass[12pt,reqno]{amsart}
\usepackage{times}

\newtheorem{thm}{Theorem}
\newtheorem{lem}[thm]{Lemma}
\newtheorem{cor}[thm]{Corollary}
\newtheorem{prop}[thm]{Proposition}

\newcommand{\C}{{\mathbb C}}
\newcommand{\D}{{\mathbb D}}

\begin{document}

\title[Logarithmic convexity of area integral means]
{\bf Logarithmic convexity of\\ area integral means for analytic functions}

\author{Chunjie Wang}
\address{Chunjie Wang, Department of Mathematics, Hebei University of Technology, 
Tianjin 300401, China}
\email{wcj@hebut.edu.cn}

\author{Kehe Zhu}
\address{Kehe Zhu, Department of Mathematics and Statistics, State University
of New York, Albany, NY 12222, USA}
\email{kzhu@math.albany.edu}

\keywords{logarithmic convexity, area integral means, Bergman spaces.}

\subjclass[2000]{Primary 30H10, 30H20}

\begin{abstract}
We show that the $L^2$ integral mean on $r\D$ of an analytic function in the unit disk $\D$
with respect to the weighted area measure $(1-|z|^2)^\alpha\,dA(z)$, where $-3\le\alpha\le0$, 
is a logarithmically convex function of $r$ on $(0,1)$. We also show that the range $[-3,0]$
for $\alpha$ is best possible.
\end{abstract}

\maketitle

\section{Introduction}

Let $\D$ denote the unit disk in the complex plane $\C$ and let $H(\D)$ denote the space
of all analytic functions in $\D$. For any real number $\alpha$ let
$$dA_\alpha(z)=(1-|z|^2)^\alpha\,dA(z),$$
where $dA$ is area measure on $\D$.

For any $f\in H(\D)$ and $0<p<\infty$ we consider the weighted area integral means
$$M_{p,\alpha}(f,r)=\frac{\displaystyle\int_{r\D}|f(z)|^p\,dA_\alpha(z)}
{\displaystyle\int_{r\D}\,dA_\alpha(z)},\quad 0<r<1.$$
It was proved in \cite{XZ} that the function $r\mapsto M_{p,\alpha}(f,r)$ is strictly
increasing for $r\in (0,1)$, unless $f$ is constant. It was also proved in \cite{XZ} that
for $\alpha\le-1$, the function $r\mapsto M_{p,\alpha}(f,r)$ is bounded on $(0,1)$ if and
only if $f$ belongs to the Hardy space $H^p$; and for $\alpha>-1$, the function
$r\mapsto M_{p,\alpha}(f,r)$ is bounded on $(0,1)$ if and only if $f$ belongs to the
weighted Bergman space 
$$A^p_\alpha=H(\D)\cap L^p(\D,dA_\alpha).$$
See \cite{D} for the theory of Hardy spaces and \cite{HKZ} for the general theory of
Bergman spaces in the unit disk.

The classical Hardy convexity theorem asserts that the integral means
$$M_p(f,r)=\frac1{2\pi}\int_0^{2\pi}|f(re^{it})|^p\,dt,$$
as a function of $r$ on $[0,1)$, is not only increasing but also logarithmically convex.
In other words, the function $r\mapsto\log M_p(f,r)$ is convex in $\log r$. See \cite{D}
again.

Motivated by Hardy's convexity theorem and by some circumstantial evidence, Xiao and Zhu boldly proposed the following conjecture in \cite{XZ}: the function $r\mapsto\log M_{p,\alpha}(f,r)$ is convex in $\log r$ when $\alpha\leq0$ and concave in $\log r$ when $\alpha>0$.

In this paper we prove the above conjecture when $-3\le\alpha\le0$ and $p=2$. The cases
$\alpha=0$ and $\alpha=-1$ are direct consequences of Hardy's convexity theorem and a
theorem of Taylor in \cite{T}; these cases were addressed in \cite{XZ}. We also show that the 
range $[-3,0]$ for $\alpha$ is best possible.

When $p=2$, we are able to reduce the problem to the case of monomials. But it should be pointed
out that the monomials are by no means simple in this problem, or at least we have not found an
easy way to deal with the monomials. Our approach is, unfortunately, by brutal force. We 
still do not know how to deal with the case $p\not=2$ for general $f$.

A great deal of elementary algebraic manipulations have been omitted in the presentation. This
would probably make the reading of the paper somewhat difficult. For those computations as well 
as the ones that remain here, we first obtained the details by hand and then verified them 
with Maple (a widely used computer algebra system).

\section{The case of monomials}

We first consider the case when $f(z)=z^k$ is a monomial. Despite the simplicity of these functions, the verification of the logarithmic convexity of $M_{p,\alpha}(z^k,r)$ is highly
nontrivial. We begin with some general lemmas concerning logarithmic convexity of positive
functions.

\begin{lem}
Suppose $f$ is twice differentiable on $(0,1)$. Then $f(x)$ is convex in $\log x$ if and only
if $f'(x)+xf''(x)\ge0$ on $(0,1)$.
\label{1}
\end{lem}

\begin{proof}
Let $t=\log x$. It follows easily from the Chain Rule that
$$\frac{d^2f}{dt^2}=x\left[f'(x)+xf''(x)\right].$$
Thus $f$ is convex in $\log x$ if and only if $f'(x)+xf''(x)\ge0$ on $(0,1)$.
\end{proof}

\begin{cor}
Suppose $f$ is twice differentiable on $(0,1)$. Then $f(x)$ is convex in $\log x$ if and only
if $f(x^2)$ is convex in $\log x$.
\label{2}
\end{cor}

\begin{proof}
For the function $g(x)=f(x^2)$, we easily compute that
$$g'(x)+xg''(x)=4x\left[f'(x^2)+x^2f''(x^2)\right].$$
The desired result then follows from Lemma~\ref{1}.
\end{proof}

\begin{cor}
Suppose $f$ is positive and twice differentiable on $(0,1)$. Then the function $\log f(x)$ is
convex in $\log x$ if and only if
$$D(f(x))=:\frac{f'(x)}{f(x)}+x\left(\frac{f'(x)}{f(x)}\right)'=\frac{f'(x)}{f(x)}
+x\frac{f''(x)}{f(x)}-x\left(\frac{f'(x)}{f(x)}\right)^2$$
is nonnegative on $(0,1)$.
\label{3}
\end{cor}

\begin{proof}
Apply Lemma~\ref{1} to the function $g(x)=\log f(x)$. The desired result follows immediately.
\end{proof}

\begin{prop}
Suppose $k\ge0$, $-2\le\alpha\le0$, and $0<p<\infty$. Then the function 
$\log M_{p,\alpha}(z^k,r)$ is convex in $\log r$.
\label{4}
\end{prop}

\begin{proof}
The case $\alpha=0$ follows from the classical Hardy convexity theorem and a theorem of Taylor
in \cite{T}; see \cite{XZ} as well as the remark after Propositoin~\ref{7}. For the rest of 
the proof we assume that $\alpha<0$.

By polar coordinates and an obvious change of variables, we have
$$M_{p,\alpha}(z^k,r)=\frac{\displaystyle\int_0^{r^2}(1-t)^\alpha t^{pk/2}\,dt}
{\displaystyle\int_0^{r^2}(1-t)^\alpha\,dt}.$$
By Corollary \ref{2}, we just need to show that the function
$$F_k(x)=\log\frac{\displaystyle\int_0^x(1-t)^\alpha t^{pk/2}\,dt}
{\displaystyle\int_0^x(1-t)^\alpha\,dt}$$
is convex in $\log x$. Rewrite
$$F_k(x)=\log f_{pk/2}(x)-\log f_0(x),\qquad 0<x<1,$$
where for any nonnegative parameter $\lambda$ we define
\begin{equation}
f_\lambda(x)=\int_0^xt^\lambda(1-t)^\alpha\,dt,\qquad 0<x<1.
\label{eq1}
\end{equation}
By Corollary \ref{3}, $F_k(x)$ is convex in $\log x$ if and only if
$$\left[\frac{f_{pk/2}'(x)}{f_{pk/2}(x)}+x\left(\frac{f_{pk/2}'(x)}{f_{pk/2}(x)}\right)'
\right]-\left[\frac{f_0'(x)}{f_0(x)}+x\left(\frac{f_0'(x)}{f_0(x)}\right)'\right]\ge0$$
for $x\in(0,1)$. Here and throughout the paper, the derivatives $f'_\lambda(x)$ and
$f''_\lambda(x)$ are taken with respect to $x$.

For any $\lambda\in[0,\infty)$ and $x\in(0,1)$ consider
\begin{equation}
\Delta(\lambda,x)=\frac{f_\lambda'(x)}{f_\lambda(x)}+x\left(\frac{f_\lambda'(x)}{f_\lambda(x)}
\right)'-\left[\frac{f_0'(x)}{f_0(x)}+x\left(\frac{f_0'(x)}{f_0(x)}\right)'\right].
\label{eq2}
\end{equation}
The convexity of $F_k(x)$ in $\log x$ is then equivalent to
$\Delta(pk/2,x)\ge0$. Since $\Delta(0,x)=0$, the desired result will follow if we can show 
that for any fixed $x\in(0,1)$, the function $\lambda\mapsto\Delta(\lambda,x)$ is increasing 
on $[0,\infty)$.

To simplify notation, we are going to write $h=f_\lambda(x)$ and use $h'$, $h''$, $h'''$ to 
denote the various derivatives of $f_\lambda(x)$ with respect to $x$. On the other hand,
the derivative of various functions with respect to $\lambda$ will be written as 
$\partial/\partial\lambda$.

Since
\begin{equation}
h=\int_0^xt^\lambda(1-t)^\alpha\,dt,
\label{eq3}
\end{equation}
we immediately obtain
\begin{equation}
h'=x^\lambda(1-x)^\alpha,\quad
h''=(\lambda-\lambda x-\alpha x)x^{\lambda-1}(1-x)^{\alpha-1}.
\label{eq4}
\end{equation}
We also have
\begin{eqnarray*}
h'''&=&x^{\lambda-2}(1-x)^{\alpha-2}\left[(-\lambda+2\lambda\alpha-\alpha+\alpha^2
+\lambda^2)x^2\right.\\
&&+\left.(-2\lambda\alpha+2\lambda-2\lambda^2)x+(\lambda^2-\lambda)\right].
\end{eqnarray*}
On the other hand, it is easy to check that
$$\frac{\partial h}{\partial\lambda}=\int_0^xt^\lambda(1-t)^\alpha\,\log t\,dt,$$
and
$$\frac{\partial h'}{\partial\lambda}=\frac{\partial}{\partial x}\left(\frac{\partial h}{\partial \lambda}\right)=h'\log x,$$
and
$$\frac{\partial h''}{\partial\lambda}=\frac{h'}{x}+h''\log x.$$

In what follows we will use the notation $A\sim B$ to denote that $A$ and $B$ have the same
sign. This differs from the customary meaning of $\sim$ but will make our presentation much easier.

Rewrite
$$\Delta(\lambda,x)=\frac{h'}{h}+x\frac{h''}{h}
-x\left(\frac{h'}{h}\right)^2-\left[\frac{f'_0}{f_0}+x\frac{f_0''}{f_0}-x\left(\frac{f_0'}{f_0}\right)^2\right].$$
Since the function inside the brackets is independent of $\lambda$, we have
\begin{eqnarray*}
\frac{\partial\Delta}{\partial\lambda}
&=&\frac{1}{h^2}\left(h\frac{\partial h'}{\partial\lambda}+xh\frac{\partial h''}{\partial \lambda}-2xh'\frac{\partial h'}{\partial\lambda}\right)-\frac{1}{h^3}\frac{\partial h}{\partial \lambda}\left(hh'+xhh''-2x(h')^2\right)\\
&=&\frac1{h^2}\left(hh'\log x+hh'+xhh''\log x-2x(h')^2\log x\right)\\
&&-\frac1{h^3}\frac{\partial h}{\partial \lambda}\left(hh'+xhh''-2x(h')^2\right)\\
&=&\frac{h'}h+\frac1{h^3}\left(h\log x-\frac{\partial h}{\partial\lambda}\right)
(hh'+xhh''-2x(h')^2).
\end{eqnarray*}

We proceed to show that 
$$\frac{\partial\Delta(\lambda,x)}{\partial\lambda}>0,\qquad \lambda>0,x\in(0,1),$$
provided that $-2\le\alpha<0$. To this end, we fix $\lambda>0$ and regard the expression
$$\frac{\partial\Delta}{\partial\lambda}=\frac{h'}h+\frac1{h^2}\left(h\log x-\frac{\partial h}{\partial\lambda}\right)(h'+xh'')\left(h-\frac{2x(h')^2}{h'+xh''}\right)$$
as a function of $x$. It is clear that $\alpha<0$ and $\lambda>0$ imply that
$$h'+xh''\sim\lambda+1-(\lambda+1+\alpha)x>0$$
for all $x\in(0,1)$.

Let us consider the following two functions (with $\lambda$ fixed again):
$$d_1(x)=h\log x-\frac{\partial h}{\partial\lambda},$$
and
$$d_2(x)=h-\frac{2x(h')^2}{h'+xh''}=h-\frac{2x^{\lambda+1}(1-x)^{\alpha+1}}
{\lambda+1-(\lambda+1+\alpha)x}.$$

Since $d_1'(x)=h/x>0$, we have $d_1(x)\ge d_1(0)=0$. By direct computations,
\begin{eqnarray*}
d_2'(x)&=&x^\lambda(1-x)^{\alpha}-\frac{2x^\lambda(1-x)^{\alpha}(\lambda+1-(\lambda+2+\alpha)x)}{\lambda+1-(\lambda+1+\alpha)x}\\
&&-\frac{2(\lambda+1+\alpha)x^{\lambda+1}(1-x)^{\alpha+1}}{(\lambda+1-(\lambda+1+\alpha)x)^2}\\
&\sim&(\lambda+1-(\lambda+1+\alpha)x)^2\\
&&-2(\lambda+1-(\lambda+2+\alpha)x)(\lambda+1-(\lambda+1+\alpha)x)\\
&&-2(\lambda+1+\alpha)x(1-x)\\
&=&-(\lambda+1)^2+2(\lambda^2+2\lambda+1+\lambda\alpha)x-(\lambda+1+\alpha)^2x^2\\
&=:&e_2(x),
\end{eqnarray*}
and
\begin{eqnarray*}
e_2'(x)&=&2(\lambda^2+2\lambda+1+\lambda\alpha)-2(\lambda+1+\alpha)^2x\\
&\ge&2(\lambda^2+2\lambda+1+\lambda\alpha)-2(\lambda+1+\alpha)^2\\
&=&-2\alpha(\lambda+2+\alpha)>0.
\end{eqnarray*}
Here and in the next paragraph again, we use the assumption that $-2\le\alpha<0$.

Note that 
$$e_2(0)=-(\lambda+1)^2<0,\quad e_2(1)=-\alpha(2+\alpha)\ge0.$$
Since $e_2(x)$ is increasing on $(0,1)$, we see that either $e_2(x)$ is always negative
on $(0,1)$ (when $\alpha=-2$) or $e_2(x)$ has exactly one zero in $(0,1)$ (when $\alpha>-2$), say $c$, so that $e_2(x)<0$ for $x\in(0,c)$ and $e_2(x)>0$ for $x\in(c,1)$.

Similarly, we observe that
$$d_2(0)=0,\quad d_2(1)=h(1)>0.$$
Here we used the convention that $+\infty>0$. In the case when $e_2(x)$ is always negative
on $(0,1)$, $d_2(x)$ is decreasing on $(0,1)$, so that $d_2(x)<d_2(0)=0$ on $(0,1)$. In the
other case, $d_2(x)$ is decreasing on $(0,c)$ and increasing on $(c,1)$, we see that $d_2(x)$ also has exactly one zero in $(0,1)$. Either way, there exists $x^*\in(0,1]$ such that
$d_2(x)>0$ when $x^*\le x<1$ and $d_2(x)<0$ when $0<x<x^*$. 

If $x^*\leq x<1$, the condition $d_2(x)>0$ implies that $\partial\Delta/\partial\lambda>0$. 
If $0<x<x^*$, the condition $d_2(x)<0$ implies that 
$$hh'+xhh''-2x(h')^2<0,$$
from which we deduce that
$$\frac{\partial\Delta}{\partial\lambda}\sim -\frac{h^2h'}{hh'+xhh''-2x(h')^2}-h\log x+\frac{\partial h}{\partial\lambda}=:\delta(x).$$

Again, it follows from direct computations that
\begin{eqnarray*}
\delta'(x)&=&-\frac{2h(h')^2+h^2h''}{hh'+xhh''-2x(h')^2}\\
&&+\frac{h^2h'(2hh''+xhh'''-3xh'h''-(h')^2)}{(hh'+xhh''-2x(h')^2)^2}-\frac{h}{x}\\
&\sim&-(2x(h')^2+xhh'')(hh'+xhh''-2x(h')^2)\\
&&+xhh'(2hh''+xhh'''-3xh'h''-(h')^2)\\
&&-(hh'+xhh''-2x(h')^2)^2\\
&=&-(hh'+2xhh'')(hh'+xhh''-2x(h')^2)\\
&&+xhh'(2hh''+xhh'''-3xh'h''-(h')^2)\\
&\sim&-\left((h')^2+xh'h''+2x^2(h'')^2-x^2h'h'''\right)h+x(h')^2(h'+xh'')\\
&\sim&\left(-(\lambda+1)^2+(2\lambda^2+4\lambda+2+2\lambda\alpha+\alpha)x-
(\lambda+1+\alpha)^2x^2\right)h\\
&&+x^{\lambda+1}(1-x)^{\alpha+1}(\lambda+1-(\lambda+1+\alpha)x)h'=:\delta_1(x).
\end{eqnarray*}
Continuing the computations, we have
\begin{eqnarray*}
\delta_1'(x)&=&[2\lambda^2+4\lambda+2+2\lambda\alpha+\alpha-2(\lambda+1+\alpha)^2x]h\\
&&-2(\lambda+1+\alpha)x^{\lambda+1}(1-x)^{\alpha+1},\\
\delta_1''(x)&=&-2(\lambda+1+\alpha)^2h+[-\alpha+2(\lambda+1+\alpha)x]
x^{\lambda}(1-x)^{\alpha},\\
\delta_1'''(x)&=&-\alpha(\lambda+(\lambda+2+\alpha)x)x^{\lambda-1}(1-x)^{\alpha-1}.
\end{eqnarray*}
Since $\alpha<0$, $\lambda>0$, and $\lambda+2+\alpha>0$, we have $\delta_1'''(x)>0$ for
all $x\in(0,1)$.

It is easy to see that $\delta_1''(0)=\delta_1'(0)=\delta_1(0)=0$. With details deferred to after the proof, we also have $\delta'(0)=0$. It then follows from elementary calculus that 
the functions $\delta_1''(x)$, $\delta_1'(x)$, $\delta_1(x)$, and $\delta'(x)$ are all 
positive on $(0,x^*)$. This shows that
$$\frac{\partial\Delta(\lambda,x)}{\partial\lambda}>0,\qquad 0<x<x^*.$$
Combining this with our earlier conclusion on $[x^*,1)$, we obtain
$$\frac{\partial\Delta(\lambda,x)}{\partial\lambda}>0,\qquad x\in(0,1).$$
In particular, for any fixed $x\in(0,1)$, the function $\lambda\mapsto\Delta(\lambda,x)$
is increasing for $\lambda\in[0,\infty)$. This completes the proof of the proposition.
\end{proof}

In the previous paragraph, we claimed that $\delta'(0)=0$. This is elementary but cumbersome, so
we deferred the details to here. Recall from the formula for $\delta'(x)$ that there are three
terms for us to consider. One of the terms is easy, namely,
$$\lim_{x\to0}\frac hx=0,$$
since we are assuming that $\lambda>0$. L'Hopital's rule gives us
$$\lim_{x\to0}\frac{h}{xh'}=\lim_{x\to0}\frac{h'}{h'+xh''}=\frac1{\lambda+1}.$$
From the explicit formulas for $h'$, $h''$, and $h'''$ we deduce that
$$\lim_{x\to0}\frac{xh''}{h'}=\lambda,\quad \lim_{x\to0}\frac{x^2h'''}{h'}=\lambda^2-\lambda.$$
Consequently,
$$\lim_{x\to0}\left[\frac h{xh'}+\frac h{xh'}\frac{xh''}{h'}-2\right]=
\frac1{\lambda+1}+\frac\lambda{\lambda+1}-2=-1.$$
Therefore,
$$\lim_{x\to0}\frac{2h(h')^2+h^2h''}{hh'+xhh''-2x(h')^2}=\lim_{x\to0}\frac{h}{x}\frac{2+\frac{h}{xh'}\frac{xh''}{h'}}{\frac{h}{xh'}+\frac{h}{xh'}\frac{xh''}{h'}-2}=0,$$
and
\begin{eqnarray*}
&&\lim_{x\rightarrow0}\frac{h^2h'(2hh''+xhh'''-3xh'h''-(h')^2)}{(hh'+xhh''-2y(h')^2)^2}\\
&&=\lim_{x\rightarrow0}\frac{h}{x}\frac{h}{xh'}\frac{\displaystyle{2\frac{h}{xh'}\frac{xh''}{h'}+\frac{h}{xh'}\frac{x^2h'''}{h'}-3\frac{xh''}{h'}-1}}{\displaystyle{\left(\frac{h}{xh'}+\frac{h}{xh'}\frac{xh''}{h'}-2\right)^2}}=0.
\end{eqnarray*}
This shows that each of three terms in the formula for $\delta'(x)$ approaches $0$ as $x\to0$.
Thus $\delta'(0)=0$.

\begin{prop}
Suppose $k\ge0$, $-3\le\alpha\le0$, and $p=2$. Then the function $\log M_{2,\alpha}(z^k,r)$
is convex in $\log r$.
\label{5}
\end{prop}

\begin{proof}
By Proposition~\ref{4}, the result already holds in the case $-2\le\alpha\le0$. So for the rest 
of the proof we assume that $-3\le\alpha<-2$.

We still consider the functions $\Delta(\lambda,x)$ and $\partial\Delta/\partial\lambda$. But
this time we restrict our attention to $0<x<1$ and $\lambda_0\le\lambda<\infty$, where
$\lambda_0=-(\alpha+2)>0$. Our strategy is to show that
$$\Delta(\lambda_0,x)>0,\qquad\frac{\partial\Delta}{\partial\lambda}(\lambda,x)>0$$
for all $x\in(0,1)$ and $\lambda\in(\lambda_0,\infty)$. This will then imply that
$$\Delta(\lambda,x)\ge\Delta(\lambda_0,x)>0$$
for all $\lambda\ge\lambda_0$ and $x\in(0,1)$. In particular, we will have $\Delta(pk/2,x)>0$
for all $k\ge1$ and $x\in(0,1)$, because in this case $p=2$ and $\lambda_0\in(0,1]$.

For $\lambda=\lambda_0$, we have
$$h=h(x)=\int_0^xt^{-2-\alpha}(1-t)^\alpha\,dt.$$
Changing variables from $t$ to $1/s$, we easily obtain
$$h(x)=-\frac1{\alpha+1}\,\left(\frac1x-1\right)^{\alpha+1}.$$
Calculating with the $D$-notation from Corollary~\ref{3} we get
$$D(h(x))=-\frac{\alpha+1}{(1-x)^2},$$
and
$$D(f_0(x))=(\alpha+1)(1-x)^{\alpha-1}\,\frac{1-x-\alpha x-(1-x)^{\alpha+1}}
{\left[1-(1-x)^{\alpha+1}\right]^2}.$$
It follows that
\begin{eqnarray*}
\Delta(\lambda_0,x)&=&D(h(x))-D(f_0(x))\\
&\sim&\left[1-(1-x)^{\alpha+1}\right]^2+(1-x)^{\alpha+1}\left[1-x-\alpha x-(1-x)^{\alpha+1}
\right]\\
&=&1-(1+x+\alpha x)(1-x)^{\alpha+1}=:\delta_3(x).
\end{eqnarray*}
Since
$$\delta_3'(x)=(\alpha+1)(\alpha+2)x(1-x)^\alpha>0,\qquad 0<x<1,$$
we see that $\delta_3(x)>\delta_3(0)=0$ for $0<x<1$. This shows that
$$\Delta(\lambda_0,x)>0,\qquad 0<x<1.$$

To finish the proof of the proposition, we indicate how to adapt the proof of
Proposition~\ref{4} to show that
$$\frac{\partial\Delta}{\partial\lambda}(\lambda,x)>0,\qquad \lambda_0<\lambda<\infty,0<x<1.$$
So for the rest of this proof, we are going to use the notation from the proof of
Proposition~\ref{4}.

First, observe that the assumptions $\lambda>\lambda_0$ and $-3\le\alpha<-2$ give
$e_2'(x)>0$ on $(0,1)$, so that $e_2(x)$ is increasing on $[0,1]$. Since
$$e_2(0)=-(\lambda+1)^2<0,\qquad e_2(1)=-\alpha(2+\alpha)<0,$$
the function $e_2(x)$ is always negative on $[0,1]$, which implies that the function $d_2(x)$
is decreasing on $(0,1)$. But $d_2(0)=0$, so $d_2(x)$ is always negative on $(0,1)$. 
Use $x^*=1$ in the proof of Proposition~\ref{4} and continue from there until the equation
$$\delta_1'''(x)=-\alpha\left[\lambda+(\lambda+2+\alpha)x\right]x^{\lambda-1}(1-x)^{\alpha-1}.$$
The assumptions $-3\le\alpha<-2$ and $\lambda>\lambda_0$ imply that $\delta_1'''(x)>0$ for
all $x\in(0,1)$. The rest of the proof of Proposition~\ref{4} remains valid here. This
completes the proof of Proposition~\ref{5}.
\end{proof}

Finally in this section we show that the range $-3\le\alpha\le0$ in the case $p=2$ is best
possible.

\begin{prop}
Suppose $\alpha\not\in[-3,0]$ and $p=2$. Then there exist positive integers $k$ such
that the function $\log M_{2,\alpha}(z^k,r)$ is not convex in $\log r$ for $r\in(0,1)$.
\label{6}
\end{prop}

\begin{proof}
Once again we consider the function $\Delta(\lambda,x)$. We are going to show that if
$\alpha\not\in[-3,0]$ then $\Delta(pk/2,x)<0$ for certain positive integers $k$ and $x$ 
sufficiently close to $1$.

First consider the case in which $\alpha>0$. In this case, we write
$$\Delta(\lambda,x)=\left(\frac{h'}h-\frac{f_0'}{f_0}\right)-x\left[\left(\frac{h'}h\right)^2-
\left(\frac{f_0'}{f_0}\right)^2\right]+x\left(\frac{h''}h-\frac{f_0''}{f_0}\right).$$
Using the formulas from (\ref{eq4}) we see that $\Delta(\lambda,x)$ is $(1-x)^{\alpha-1}$ times
the sum of
\begin{equation}
(1-x)\left(\frac{x^\lambda}h-\frac1{f_0}\right)-x(1-x)^{\alpha+1}\left(\frac{x^{2\lambda}}{h^2}
-\frac1{f_0^2}\right)
\label{eq5}
\end{equation}
and
\begin{equation}
\frac x{hf_0}\left[(\lambda-\lambda x-\alpha x)x^{\lambda-1}f_0+\alpha h\right].
\label{eq6}
\end{equation}
The assumption $\alpha>0$ implies that the integrals
$$h(1)=\int_0^1t^\lambda(1-t)^\alpha\,dt,\quad f_0(1)=\int_0^1(1-t)^\alpha\,dt,$$
are finite and positive numbers. It follows that the function defined by (\ref{eq5}) approaches 
$0$ as $x\to1$, and
$$\lim_{x\to1}\left[(\lambda-\lambda x-\alpha x)x^{\lambda-1}f_0+\alpha h\right]=-\alpha
\int_0^1(1-t^\lambda)(1-t)^\alpha\,dt<0.$$
We deduce that $\Delta(\lambda,x)<0$ when $x$ is sufficiently
close to $1$. Consequently, if $\alpha>0$, then for any $0<p<\infty$ and any $k>0$, the function 
$\log M_{p,\alpha}(z^k,r)$ is not convex in $\log r$ for $r\in(0,1)$.

Next we consider the case in which $\alpha<-3$. In this case, we rewrite
$$\Delta(\lambda,x)=\left(\frac{h'}{h}-\frac{f_0'}{f_0}\right)+x\left(\frac{h''}{h}-2\frac{h'f_0'}{hf_0}-\frac{f_0''}{f_0}+2\frac{(f_0')^2}{f_0^2}\right)-x\left(\frac{h'}{h}-\frac{f_0'}{f_0}\right)^2.$$
Let
$$\Delta_1(\lambda,x)=\left(\frac{h'}{h}-\frac{f_0'}{f_0}\right)-x\left(\frac{h'}{h}-\frac{f_0'}{f_0}\right)^2,$$
and
$$\Delta_2(\lambda,x)=\frac{h''}{h}-2\,\frac{h'f_0'}{hf_0}-\frac{f_0''}{f_0}+2\,\frac{(f_0')^2}{f_0^2}.$$
Then 
$$\Delta(\lambda,x)=\Delta_1(\lambda,x)+x\Delta_2(\lambda,x).$$

Observe that the condition $\alpha<-3$ implies that
$$h-x^\lambda f_0=\int_0^x(t^\lambda-x^\lambda)(1-t)^\alpha\,dt\to-\infty$$
as $x\to1$, and we can use L'Hospital's Rule to obtain the limits
\begin{equation}
\lim_{x\to1}\frac{(1-x)^{\alpha+1}}{h}=\lim_{x\to1}\frac{(1-x)^{\alpha+1}}{f_0}=-(\alpha+1),
\label{eq7}
\end{equation}
and
$$\lim_{x\to1}\frac{h-x^\lambda f_0}{(1-x)^{\alpha+2}}=\lim_{x\to1}\frac{-\lambda x^{\lambda-1} f_0}{-(\alpha+2)(1-x)^{\alpha+1}}=-\frac{\lambda}{(\alpha+1)(\alpha+2)},$$
and
$$\lim_{x\to1}\left(\frac{h'}{h}-\frac{f_0'}{f_0}\right)=\lim_{x\to1}\frac{(1-x)^{2\alpha+2}}{hf_0}\cdot\frac{x^\lambda f_0-h}{(1-x)^{\alpha+2}}=\lambda\,\frac{\alpha+1}{\alpha+2}.$$
It follows that
$$\lim_{x\to1}\Delta_1(\lambda,x)=\lambda\,\frac{\alpha+1}{\alpha+2}-\left(\lambda\,
\frac{\alpha+1}{\alpha+2}\right)^2=\lambda\,\frac{\alpha+1}{\alpha+2}\left(1-\lambda\,
\frac{\alpha+1}{\alpha+2}\right).$$

On the other hand,
$$\Delta_2(\lambda,x)=\frac{(h''f_0-2h'f'_0)f_0+(2(f_0')^2-f_0f''_0)h}{hf_0^2},$$
which is $(1-x)^{3(\alpha+1)}/(hf_0^2)$ times the sum of
$$\frac{x^{\lambda-1}\left[\lambda-\lambda x-\alpha x-(\lambda-\lambda x+\alpha x +2x)(1-x)^{\alpha+1}\right]}{(\alpha+1)(1-x)^{2\alpha+4}}\,f_0$$
and
$$\frac{(\alpha+2)(1-x)^{\alpha+1}+\alpha}{(\alpha+1)(1-x)^{2\alpha+4}}\,h.$$
We rearrange terms to obtain
$$\Delta_2(\lambda,x)=\frac{(1-x)^{3(\alpha+1)}}{(\alpha+1)hf_0^2}\,\left[T_1(\lambda,x)+T_2(\lambda,x)\right],$$
where
$$T_1(\lambda,x)=\frac{\lambda x^{\lambda-1}}{(1-x)^{\alpha+2}}\frac{f_0}{(1-x)^{\alpha+1}}+\frac{\alpha}{(1-x)^{\alpha+2}}\frac{h-x^\lambda f_0}{(1-x)^{\alpha+2}},$$
and
$$T_2(\lambda,x)=\frac{(\alpha+2)h-(\lambda-\lambda x+\alpha x +2x)x^{\lambda-1}f_0}{(1-x)^{\alpha+3}}.$$

It follows from (\ref{eq7}) that
$$\lim_{x\to1}\frac{(1-x)^{3(\alpha+1)}}{(\alpha+1)hf_0^2}=-(\alpha+1)^2,$$
and
$$\lim_{x\to1}T_1(\lambda,x)=0.$$
Since $(\alpha+1)f_0=1-(1-x)^{\alpha+1}$ and $\alpha<-3$, it follows from L'Hopital's rule 
and elementary manipulations that
$$\lim_{x\to1}T_2(\lambda,x)=-\frac{\lambda(\lambda-1)}{(\alpha+1)(\alpha+3)}.$$
Therefore,
\begin{eqnarray*}
\lim_{x\to1}\Delta(\lambda,x)&=&\lambda\,\frac{\alpha+1}{\alpha+2}\left(1-\lambda\,
\frac{\alpha+1}{\alpha+2}\right)+\lambda(\lambda-1)\frac{\alpha+1}{\alpha+3}\\
&=&\frac{\lambda(\alpha+1)(\lambda+2+\alpha)}{(\alpha+2)^2(\alpha+3)}.
\end{eqnarray*}
If $p=2$ and $k=1$, then for $\lambda=pk/2=1$ we have
$$\lim_{x\to1}\Delta(\lambda,x)=\frac{\alpha+1}{(\alpha+2)^2}<0.$$
This shows that $\Delta(\lambda,x)<0$ for $x$ sufficiently close to $1$. Thus the
function $\log M_{2,\alpha}(z,r)$ is not convex in $\log r$.
\end{proof}

The careful reader will notice that our methods in this section can be applied to determine
the optimal range for $\alpha$ when $p\ge2$ and the function is a monomial. With a little extra
work this can probably be pushed to work for $p<2$ as well. Since we are unable to pass from the
monomials to a general function in the case when $p\not=2$, we will not pursue this issue here
further.

\section{The case of $p=2$ and arbitrary $f$}

In this section we prove the logarithmic convexity of $M_{p,\alpha}(f,r)$ when $p=2$ and
$-3\le\alpha\leq0$. Basically, we reduce the problem to the case of monomials using a
theorem of Taylor from \cite{T}. More specifically, the following proposition follows from
\cite{T} using Banach space techniques. But an elementary proof is provided here using only
Taylor expansions and the Cauchy-Schwarz inequality.

\begin{prop}
Suppose $\{h_k(x)\}$ is a sequence of positive and twice differentiable functions on $(0,1)$
such that the function
$$H(x)=\sum_{k=0}^\infty h_k(x)$$
is also twice differentiable on $(0,1)$. If for each $k$ the function $\log h_k(x)$ is convex 
in $\log x$, then $\log H(x)$ is also convex in $\log x$.
\label{7}
\end{prop}

\begin{proof}
Recall from Corollary~\ref{3} that $\log f(x)$ is convex in $\log x$ if and only if
$$D(f(x))=\frac{f'(x)}{f(x)}+x\frac{f''(x)}{f(x)}-x\left(\frac{f'(x)}{f(x)}\right)^2\ge0.$$
By the Cauchy-Schwarz inequality,
\begin{eqnarray*}
\left[H'(x)\right]^2&=&\left[\sum_{k=0}^\infty h_k'(x)\right]^2
=\left[\sum_{k=0}^\infty\sqrt{h_k(x)}\frac{h'_k(x)}{\sqrt{h_k(x)}}\right]^2\\
&\le&\sum_{k=0}^\infty h_k(x)\sum_{k=0}^\infty\frac{h_k'(x)^2}{h_k(x)}
=H(x)\sum_{k=0}^\infty\frac{h_k'(x)^2}{h_k(x)}.
\end{eqnarray*}
From this we deduce that
\begin{eqnarray*}
D(H(x))&=&\frac{H'(x)}{H(x)}+x\frac{H''(x)}{H(x)}-x\left(\frac{H'(x)}{H(x)}\right)^2\\
&\ge&\frac{H'(x)}{H(x)}+x\frac{H''(x)}{H(x)}-x\,\frac{\sum_{k=0}^\infty h_k'(x)^2/h_k(x)}{H(x)}\\
&=&\frac1{H(x)}\left[H'(x)+xH''(x)-x\sum_{k=0}^\infty\frac{h_k'(x)^2}{h_k(x)}\right]\\
&=&\frac1{H(x)}\sum_{k=0}^\infty\left[h_k'(x)+xh_k''(x)-x\,\frac{h_k'(x)^2}{h_k(x)}\right]\\
&=&\frac1{H(x)}\sum_{k=0}^\infty h_k(x)D(h_k(x)).
\end{eqnarray*}
By assumption, each $D(h_k(x))$ is nonnegative on $(0,1)$ and each $h_k$ is positive on
$(0,1)$, so $D(H(x))$ is nonnegative on $(0,1)$. This completes the proof of the proposition.
\end{proof}

It is easy to adapt the above proof to obtain a continuous version of the proposition in terms
of integrals instead of infinite series. The resulting version is also a special case of the more general theorem in \cite{T}, so we omit the details.

We now obtain the main result of the paper.

\begin{thm}
Suppose $f\in H(\D)$ and $-3\le\alpha\le0$. Then the function $r\mapsto\log M_{2,\alpha}(f,r)$
is convex in $\log r$. Moreover, the range $-3\le\alpha\le0$ is best possible.
\label{8}
\end{thm}

\begin{proof}
Suppose
$$f(z)=\sum_{k=0}^\infty a_kz^k.$$
It follows from integration in polar coordinates that
$$M_{2,\alpha}(f,r)=\sum_{k=0}^\infty|a_k|^2M_{2,\alpha}(z^k,r).$$
By Proposition~\ref{5}, each function
$$h_k(r)=|a_k|^2M_{2,\alpha}(z^k,r)$$
has the property that $\log h_k(r)$ is convex in $\log r$. So by Proposition~\ref{7}, the
function $\log M_{2,\alpha}(f,r)$ is convex in $\log r$.

That the range $-3\le\alpha\le0$ is best possible follows from Proposition~\ref{6}.
\end{proof}

\section{Two Examples}

It was shown in \cite{XZ} by an example that when $\alpha>0$, $\log M_{p,\alpha}(f,r)$ is
not always convex in $\log r$. Based on this particular example and some circumstantial evidence,
it was further conjectured in \cite{XZ} that if $\alpha>0$, the function $\log M_{p,\alpha}(f,r)$ is concave in $\log r$. We show in this section that this is not so. In
fact, when $\alpha=1$ or $\alpha=-4$, we give examples such that the function $\log M_{2,\alpha}(f,r)$ is {\it neither} convex {\it nor} concave on $(0,1)$. These examples also
illustrate the somewhat abstract calculations we did in Section 2 with arbitrary monomials.

First, let $p=2$, $\alpha=1$, and $f(z)=1+z$. Then
$$dA_1(z)=(1-|z|^2)dA(z),\ z\in\D,$$ 
and 
$$A_1(r\D)=\int_{r\D}(1-|z|^2)dA(z)=\frac\pi2\,r^2(2-r^2).$$
Also,
\begin{eqnarray*}
\int_{r\D}|f(z)|^2\,dA_1(z)&=&\int_{r\D}(1+|z|^2)(1-|z|^2)\,dA(z)\\
&=&2\pi\int_0^r(1+t^2)(1-t^2)t\,dt\\
&=&\frac\pi3\,r^2(3-r^4).
\end{eqnarray*}
It follows that
$$M_{2,1}(1+z,r)=\frac{2(3-r^4)}{3(2-r^2)}.$$
By Corollary~\ref{2}, we just need to consider the convexity of the following function
in $\log x$:
$$h(x)=\frac{3-x^2}{2-x},\qquad 0<x<1.$$
Using the $D$-notation from Corollary~\ref{3}, we have
$$D(h(x))=D(3-x^2)-D(2-x).$$
It is elementary to show that
$$D(3-x^2)=-\frac{12x}{(3-x^2)^2},\quad D(2-x)=-\frac2{(2-x)^2},$$
from which we deduce that
$$D(h(x))=\frac{2g(x)}{(2-x)^2(3-x^2)^2},$$
where
$$g(x)=9-24x+18x^2-6x^3+x^4.$$

For $x\in(0,1)$ we have
$$g'(x)=-24+36x-18x^2+4x^3,$$
and
$$g''(x)=36-36x+12x^2>0.$$
Thus $g'(x)$ is increasing on $[0,1]$. In particular,
$$g'(x)\le g'(1)=-2<0,\qquad 0\le x\le1.$$
This shows that $g(x)$ is decreasing on $[0,1]$. Since $g(0)=9>0$ and $g(1)=-2<0$, there
exists a point $c\in(0,1)$ such that $g(x)>0$ for $x\in(0,c)$ and $g(x)<0$ for $x\in(c,1)$.
Thus the function $\log h(x)$ is neither convex nor concave in $\log x$.

We note that the functions $z+a$ have also been considered by Xiao and Xu \cite{XX} in their 
recent work on weighted area integral means of analytic functions and other related problems.

Next, consider the case when $p=2$, $\alpha=-4$, and $f(z)=\sqrt 2\,z$. In this case,
$$dA_{-4}(z)=(1-|z|^2)^{-4}dA(z),\ z\in\D,$$ 
and 
$$A_{-4}(r\D)=\int_{r\D}(1-|z|^2)^{-4}dA(z)=\frac{\pi}{3}\left((1-r^2)^{-3}-1\right).$$
Also,
\begin{eqnarray*}
\int_{r\D}|f(z)|^2\,dA_{-4}(z)&=&\int_{r\D}2|z|^2(1-|z|^2)^{-4}\,dA(z)\\
&=&\frac{\pi}{3}\left(2r^2(1-r^2)^{-3}-(1-r^2)^{-2}+1\right).
\end{eqnarray*}
It follows that
$$M_{2,-4}(z,r)=\frac{3r^2-r^4}{3-3r^2+r^4}.$$
By Corollary~\ref{2}, we just need to consider the convexity of the following function
in $\log x$:
$$h(x)=\frac{3x-x^2}{3-3x+x^2},\qquad 0<x<1.$$
Direct computations show that
\begin{eqnarray*}
h'(x)&=&\frac{3(3-2x)}{(3-3x+x^2)^2},\\
h''(x)&=&\frac{18(1-x)(2-x)}{(3-3x+x^2)^3}.\\
\end{eqnarray*}
Using the $D$-notation from Corollary~\ref{3}, we have
$$D(h(x))\sim 18-36x+21x^2-4x^3=:g(x).$$

For $x\in(0,1)$ we have
$$g'(x)=-36+42x-12x^2,$$
and
$$g''(x)=42-24x>0.$$
Thus $g'(x)$ is increasing on $[0,1]$. In particular,
$$g'(x)\le g'(1)=-4<0,\qquad 0\le x\le1.$$
This shows that $g(x)$ is decreasing on $[0,1]$. Since $g(0)=18>0$ and $g(1)=-1<0$, there
exists a point $c\in(0,1)$ such that $g(x)>0$ for $x\in(0,c)$ and $g(x)<0$ for $x\in(c,1)$.
Thus the function $\log h(x)$ is neither convex nor concave in $\log x$.

\end{document}